\documentclass[12pt,a4paper]{amsart}
\usepackage[utf8]{inputenc} 
\usepackage[top=3cm, bottom=3.5cm, left=2.5cm, right=2.5cm]{geometry}
\usepackage{amsmath,amsthm,amsfonts,amstext,amssymb}
\usepackage{mathrsfs}
\usepackage{enumerate}
\usepackage{hyperref}
\usepackage{dsfont}
\usepackage{color}
\usepackage{xcolor}
\hypersetup{
    colorlinks,
    linkcolor={blue!50!black},
    citecolor={blue!50!black},
    urlcolor={blue!80!black}
}

\def\R{{\mathbb{R}}}
\def\N{{\mathbb{N}}}
\def\Z{{\mathbb{Z}}}
\def\P{{\mathbb P}} 
\def\E{{\mathbb E}}

\newtheorem{theorem}{Theorem}[section]

\newtheorem{lemma}[theorem]{Lemma}
\newtheorem{proposition}[theorem]{Proposition}

\theoremstyle{definition}

\newtheorem{definition}[theorem]{Definition}
\newtheorem{remark}[theorem]{Remark}

\numberwithin{equation}{section}

\begin{document}

\title[Uniqueness for the vacant set of random interlacements]{Uniqueness of the infinite connected component for the vacant set of random interlacements on amenable transient graphs}

\author{Yingxin Mu}
\address{
  Yingxin Mu,
  University of Leipzig, Institute of Mathematics,
  Augustusplatz 10, 04109 Leipzig, Germany.
}
\email{yingxin.mu@uni-leipzig.de}

\author{Artem Sapozhnikov}
\address{
  Artem Sapozhnikov,
  University of Leipzig, Institute of Mathematics,
  Augustusplatz 10, 04109 Leipzig, Germany.
}
\email{artem.sapozhnikov@math.uni-leipzig.de}\thanks{The research of both authors has been supported by the DFG Priority Program 2265 ``Random Geometric Systems'' (Project number 443849139).}

\begin{abstract}
We prove the uniqueness of the infinite connected component for the vacant set of random interlacements on general vertex-transitive amenable transient graphs. 
Our approach is based on connectedness of random interlacements and differs from the one used by Teixera \cite{Teixeira-Uniqueness} to prove the uniqueness of the infinite connected component for the vacant set of random interlacements on $\Z^d$.
\end{abstract}

 
\maketitle

\section{Introduction}
The model of random interlacements was introduced by Sznitman in \cite{Sznitman-AM} on lattices $\Z^d$ ($d\geq 3$) and defined on arbitrary transient graphs by Teixeira in \cite{Teixeira-RI}. The vacant set of random interlacements has been an important example of percolation model with strong, algebraically decaying correlations, see e.g.\ \cite{CT-Book, DRS-Book, DGRS-GFF} for comprehensive literature review. 
In this paper we study the vacant set of random interlacements on infinite vertex-transitive amenable transient graphs. 

\smallskip

Let $G=(V,E)$ be a locally finite graph. $G$ is \emph{vertex-transitive} if for any $x,y\in V$ there exists a graph automorphism of $G$, which maps $x$ to $y$. $G$ is \emph{amenable} if the vertex isoperimetric constant $\kappa_V(G)$, defined by 
\[
\kappa_V(G)=\inf\Big\{\frac{|\partial A|}{|A|}\,:\,|A|<\infty\Big\},
\]
is equal to $0$, where $\partial A = \{x\in A\,:\,(x,y)\in E\text{ for some }y\notin A\}$ is the (inner) vertex boundary of $A$. $G$ is \emph{transient} if a simple random walk on $G$ is transient. 

\smallskip

Random interlacements on $G$ is the range of a Poisson cloud of doubly-infinite random walks on $G$, whose density is controlled by a parameter $u>0$. 
We postpone its precise definition to Section~\ref{sec:RI} and just mention here that the law of random interlacements $\mathcal I^u$---as a random subset of $V$---is uniquely characterized by the relations 
\begin{equation}\label{def:Iu-1}
\P[\mathcal I^u\cap K=\emptyset] = e^{-u\mathrm{cap}(K)},\quad \text{for every finite }K\subset V,
\end{equation}
where $\mathrm{cap}(K)$ stands for the random walk capacity of $K$. The \emph{vacant set of random interlacements at level $u$} is defined as the complement of $\mathcal I^u$:
\[
\mathcal V^u = V\setminus \mathcal I^u.
\]
Our main result concerns the number of infinite connected components of $\mathcal V^u$. 

\begin{theorem}\label{thm:uniqueness-RI}
Let $G$ be a vertex-transitive amenable transient graph. For any $u>0$, the number of infinite connected components in $\mathcal V^u$ is either a.s.\ equal to $0$ or a.s.\ equal to $1$.
\end{theorem}
The main challenge to prove uniqueness for the vacant set of random interlacements is lack of the so-called \emph{finite energy property}. 
It is not clear how to overcome this difficulty using only the description of random interlacements by \eqref{def:Iu-1}. In the previous works \cite{Sznitman-AM,Teixeira-Uniqueness}, the characterization of random interlacements as the range of a Poisson cloud of random walks (see \eqref{eq:RI-definition}) and careful rerouting of the random walks enabled to overcome the obstruction, but the method was very sensitive to geometry of the ambient graph.
In particular, the result of Theorem~\ref{thm:uniqueness-RI} was obtained by Teixeira in \cite{Teixeira-Uniqueness} in the case of $\Z^d$ ($d\geq 3$), but his proof uses the structure of $\Z^d$ in a crucial way and cannot be extended to general graphs. 
To explain the novelty of our approach, let us briefly discuss the common strategy for the proof of uniqueness. 

Let $N$ be the number of infinite connected components in $\mathcal V^u$. 
By ergodicity (see \cite{TT-RI-Connected} and Proposition~\ref{prop:TT}), $N$ is constant almost surely. 
If $N=k$ a.s.\ for some $2\leq k<\infty$, then there is a sufficiently large finite set $K$, which intersects all $k$ infinite components with positive probability. 
If one can change such configurations locally at ``finite cost'' to merge all the infinite components into one, then one obtains $N=1$ with positive probability and arrives at a contradiction. If $N=\infty$ a.s., the Burton-Keane argument \cite{BK-Uniqueness} leads to a contradiction by showing that there is a positive density of so-called trifurcations---locations where an infinite component locally splits into at least $3$ infinite components. 
The existence of trifurcations is also proved by a local modification argument on the paths that intersect a large finite set $K$. In previous works \cite{Sznitman-AM,Teixeira-Uniqueness}, the behavior of the random walk paths that visit $K$ was modified \emph{inside} $K$ in order to assemble a desired configuration in $K$, which is very sensitive to the local geometry of the ambient graph. Our method is based on the \emph{connectedness} of $\mathcal I^u$ (see \cite{TT-RI-Connected} and Proposition~\ref{prop:TT}); more precisely, we reroute the random walks from the Poisson cloud that ever visit $K$ locally between their first and last visit to $K$ through the random interlacement $\mathcal I^u$ \emph{outside} of $K$, see the proofs of Proposition~\ref{prop:N=k} and Lemma~\ref{l:trifurcation}. In \cite{MS-BI}, we use similar local modification in the proof of uniqueness of the infinite connected component for the vacant set of Brownian interlacements. 

\smallskip

An immediate application of Theorem~\ref{thm:uniqueness-RI} is the continuity of the percolation function $\theta(u)$---the probability that the connected component of a given vertex in $\mathcal V^u$ is infinite---in the supercritical phase of the vacant set of random interlacements. The proof follows a by now standard argument of van den Berg and Keane \cite{BK-Continuity} (see e.g.\ \cite[Corollary~1.2]{Teixeira-Uniqueness} for an application to the vacant set of random interlacements on $\Z^d$).

\smallskip

The rest of the paper is organized as follows. Section~\ref{sec:RI} contains definition of random interlacements point process as a Poisson point process on the space of doubly infinite paths (see Section~\ref{sec:RIPP}) as well as some useful sampling procedure for the paths that visit a finite set (see \eqref{eq:sample-1} and \eqref{eq:sample-2}). The random interlacements at level $u$ is defined in \eqref{eq:RI-definition} as the range of all paths from the interlacement point process. We prove Theorem~\ref{thm:uniqueness-RI} in Section~\ref{sec:uniqueness-proof}. A key ingredient for the proof is Proposition~\ref{prop:IKn-connected}, in which we strengthen the result of Teixeira and Tykesson (see Proposition~\ref{prop:TT}) about connectedness of random interlacements.

\section{Random interlacements}\label{sec:RI}

In this section we define random interlacements on $G$ precisely. 
For further details and proofs we refer the reader to \cite{Teixeira-RI}.

\smallskip

Let $\mathsf W_+$ be the space of nearest neighbor paths $w:\N_0\to V$ on $G$ such that $w(n)\to\infty$ for $n\to\infty$\footnote{$w(n)\to\infty$ if for any finite set $K\subset V$ there exists $n_K$ such that $w(n)\notin K$ for all $n\geq n_K$.}.
We denote by $X_n$, $n\geq 0$, the canonical process on $\mathsf W_+$ (i.e.\ $X_n(w)=w(n)$) and by $\mathcal W_+$ the sigma-algebra on $\mathsf W_+$ generated by the canonical process. Since $G$ is transient, the law of the simple random walk on $G$ started from $x\in V$, denoted by $\mathsf P_x$, is a probability measure on $(\mathsf W_+,\mathcal W_+)$.

\smallskip

We write $p_n(x,x') = \mathsf P_x[X_n=x']$ for the $n$-step transition probability of the random walk and $g(x,x') = \sum\limits_{n=0}^\infty p_n(x,x')$ for the respective Green function. 

\smallskip

For a finite set $K$ in $V$, we define the equilibrium measure of $K$ by 
\[
e_K(x) =\mathsf P_x\big[X_n\notin K\text{ for all }n\geq 1\big].
\]
Note that $e_K$ is supported on $\partial K$. The total mass of $e_K$ is called the capacity of $K$ and is denoted by $\mathrm{cap}(K)$. 
Let $L_K=\sup\{n\geq 0\,:\, X_n\in K\}$ be the last visit time in $K$. Then the joint law of $L_K$ and $X_{L_K}$ (under $\mathsf P_x$) is given by 
\begin{equation}\label{eq:lastvisit}
\mathsf P_x[L_K=n,X_{L_K}=y] = p_n(x,y)e_K(y).
\end{equation}

For $x\in\partial K$, we denote by $\mathsf P_x^K$ the law of the simple random walk starting from $x$ and conditioned on staying outside of $K$ for all $n\geq 1$. We denote by $\mathsf P_{x,y}^n$ the random walk bridge measure in time $n\geq 0$ from $x$ to $y$.
Then
\begin{equation}\label{eq:conditional-independence-exit-time}
\begin{array}{c}
\text{under $\mathsf P_x$, conditionally on $L_K=n$ and $X_{L_K}=y\in \partial K$, the processes}\\ 
\text{$(X_s)_{0\leq s\leq n}$ and $(X_{L_K+s})_{s\geq 0}$ are independent and have laws $\mathsf P^n_{x,y}$ resp.\ $\mathsf P^K_y$.}
\end{array}
\end{equation}

\subsection{Compatible measures on doubly-infinite paths}
Let $\mathsf W$ be the space of doubly-infinite nearest neighbor paths $w:\Z\to V$ tending to infinity at positive and negative infinite times. We denote by $X_n$, $n\in\Z$, the canonical process on $\mathsf W$  (i.e.\ $X_n(w) = w(n)$) and by $\mathcal W$ the sigma-algebra on $\mathsf W$ generated by the canonical process. We denote the canonical time shift on $\mathsf W$ by $\theta_n$, $n\in\Z$. For a finite set $K\subset V$, we define the first entrance time of $w\in\mathsf W$ in $K$ as 
$H_K(w)= \inf\{n\in\Z\,:\,X_n(w)\in K\}$, and write
\[
\mathsf W_K = \{w\in \mathsf W\,;\,H_K(w)<\infty\},\quad 
\mathsf W_K^0 = \{w\in \mathsf W\,:\, H_K(w) = 0\}
\]
for the sets of paths that ever visit $K$, resp., visit $K$ for the first time at time $0$. 

\smallskip

Consider the following measure on $\mathsf W_K^0$:
\begin{equation}\label{def:Q_K}
Q_K\big[(X_{-n})_{n\geq 0}\in A,\,X_0=x,\,(X_n')_{n'\geq 0}\in A'\big] = \mathsf P^K_x[A]\,e_K(x)\,\mathsf P_x[A'],\quad A,A'\in \mathcal W_+.
\end{equation}
The measures $Q_K$ are compatible, in the sense that $Q_K = \theta_{H_K}\circ\big(\mathds{1}_{\mathsf W_K}\,Q_{K'}\big)$, for any finite sets $K$ and $K'$ with $K\subseteq K'$, see e.g.\ the proof of \cite[Theorem~2.1]{Teixeira-RI}. Note that $Q_K$ is a finite measure with $Q_K[\mathsf W_K^0] = \mathrm{cap}(K)$. 

\medskip

For $K\subseteq V$ and $w\in \mathsf W_K$, we define the last visit time of $w$ in $K$ by 
\[
L_K(w) = \sup\{n\in \Z\,:\,X_n(w)\in K\}. 
\]
By \eqref{eq:lastvisit}, \eqref{eq:conditional-independence-exit-time} and \eqref{def:Q_K}, for any finite $K$ and $A,A'\in\mathcal W_+$, 
\begin{multline}\label{eq:RWB-representation-Q_K}
Q_K\big[(X_{-s})_{s\geq 0}\in A,\, (X_s)_{0\leq s\leq L_K}\in\cdot,\, (X_{L_K+s})_{s\geq 0}\in A'\big]\\
= 
\sum\limits_{x,x'\in\partial K} e_K(x)\,e_K(x')\,g(x,x')\,
\mathsf P^K_x[A]\,\Big(\sum\limits_{n=0}^\infty \frac{p_n(x,x')}{g(x,x')}\,\mathsf P^n_{x,x'}[\cdot]\Big)\,\mathsf P^K_{x'}[A'],
\end{multline}
where $(X_s)_{0\leq s\leq L_K}$ is viewed as a random element on the space $\mathsf W_{\text{fin}}$ of nearest neighbor paths of finite duration.
The identity \eqref{eq:RWB-representation-Q_K} states that under $Q_K$, the pieces of the random path before the first entrance time in $K$, after the last visit time in $K$, and between those times are conditionally independent, given the locations of the first and last visits of the path in $K$.

\subsection{Random interlacement measure}
We now define a suitable sigma-finite measure on doubly-infinite paths, whose restriction to every $\mathsf W_K$ is $Q_K$. 

\smallskip

Two paths $w$ and $w'$ in $\mathsf W$ are called equivalent, if $w'=\theta_n(w)$ for some $n\in\Z$. The quotient set of $\mathsf W$ modulo this equivalence relation is denoted by $\mathsf W^*$. The canonical projection $\pi^*:\mathsf W\to \mathsf W^*$ induces the sigma-algebra $\mathcal W^* = \{A\subseteq \mathsf W^*\,:\,(\pi^*)^{-1}(A)\in \mathcal W\}$ on $\mathsf W^*$. For a finite set $K$ in $V$, we denote by $\mathsf W_K^*$ the image of $\mathsf W_K$ under $\pi^*$. 
Note that $\pi^*$ maps bijectively $\mathsf W_K^0$ onto $\mathsf W_K^*$.

\smallskip

By \cite[Theorem~2.1]{Teixeira-RI}, there exists a unique sigma-finite measure $\nu$ on $(\mathsf W^*,\mathcal W^*)$, whose restriction to any $\mathsf W_K^*$ coincides with $Q_K$, more precisely, 
\[
\mathds{1}_{\mathsf W_K^*}\,\nu = \pi^*\circ Q_K,\quad \text{for any finite $K\subset V$}.
\]
Note that $\nu[\mathsf W^*_K] = Q_K[\mathsf W_K] = \mathrm{cap}(K)$. 

\subsection{Random interlacement point process}\label{sec:RIPP}

Consider the space of point measures 
\[
\Omega = \Big\{\omega = \sum\limits_{i\geq 1} \delta_{(w_i^*,u_i)}\,:\,\omega\big(\mathsf W_K^*\times[0,u]\big)<\infty\text{ for all finite $K\subset V$ and $u>0$}\Big\}
\]
on $\mathsf W^*\times\R_+$, endowed with the sigma-algebra $\mathcal A$ generated by the evaluation maps 
\[
\omega\mapsto\omega(E), \quad E\in\mathcal W^*\otimes\mathcal B(\R_+),
\]
and denote by $\P$ the Poisson point measure on $\mathsf W^*\times\R_+$ with intensity $\nu\otimes du$; the random point measure with law $\P$ is called the \emph{random interlacement point process on $G$}. 

The random variable 
\[
N_{K,u} = N_{K,u}(\omega) = \omega\big(\mathsf W_K^*\times[0,u]\big),
\]
which counts the number of trajectories with labels $\leq u$ (in $\omega$) that visit $K$, has Poisson distribution with parameter $u\mathrm{cap}(K)$. 

\smallskip

For any finite $K\subset V$, given $N_{K,u} = n$, the $n$ trajectories of the random interlacement point process that visit $K$ and have labels $\leq u$ are independent random elements of $\mathsf W^*_K$ with the common distribution $\frac{1}{\mathrm{cap}(K)}\big(\pi^*\circ Q_K\big)$, whose labels are independent uniformly distributed on $[0,u]$. By \eqref{eq:RWB-representation-Q_K}, each of them can be sampled (independently) as follows: 
\begin{itemize}
\item
Sample the locations of the first entrance and last visit in $K$, $(X_i,X_i')$, from the distribution 
\begin{equation}\label{eq:sample-1}
\frac{1}{\mathrm{cap}(K)}\,g(x,x')\,e_K(x)\,e_K(x'),\quad x,x'\in\partial K;
\end{equation}
\item
Given $X_i=x_i$ and $X_i'=x_i'$, sample independently random paths $\gamma_i$ and $\gamma_i'$ in $\mathsf W_+$ and $\widetilde \gamma_i$ in $W_{\mathrm{fin}}$ respectively from the distributions
\begin{equation}\label{eq:sample-2}
\mathsf P^K_{x_i},\quad \mathsf P^K_{x_i'},\quad \sum\limits_{n=0}^\infty \frac{p_n(x,x')}{g(x,x')}\,\mathsf P^n_{x,x'}[\cdot];
\end{equation}
\item
Let $w_i$ be the concatenation of the time reversal of $\gamma_i$, $\widetilde \gamma_i$ and $\gamma_i'$, so that $w_i(n) = \gamma_i(-n)$ for $n\leq 0$. (Note that $w_i$ is a random path in $\mathsf W_K^0$ with the law $\frac{1}{\mathrm{cap}(K)}Q_K$.)
\item
To get the desired random element of $\mathsf W^*_K\times[0,u]$, we project $w_i$ onto $\mathsf W^*$ and assign it an independent label from the uniform distribution on $[0,u]$. 
\end{itemize}

\smallskip

For any $u>0$, the random point measure $\iota^u$ on $\mathsf W^*$, defined by 
\[
\omega = \sum\limits_{i\geq 1}\delta_{(w_i^*,u_i)}\,\,\mapsto\,\, \iota^u(\omega) = \sum\limits_{i\geq 1:\,u_i\leq u}\delta_{w_i^*},
\]
is called \emph{random interlacement point process on $G$ at level $u$}. Note that (under $\P$) $\iota^u$ is a Poisson point process on $\mathsf W^*$ with intensity $u\nu$.

\subsection{Random interlacement}\label{sec:RIset}
For any $u>0$, the \emph{random interlacement at level $u$} is defined as
\begin{equation}\label{eq:RI-definition}
\mathcal I^u(\omega) = \bigcup\limits_{i\geq 1:\,u_i\leq u}
\text{range}(w_i^*),\quad \omega = \sum\limits_{i\geq 1}\delta_{(w_i^*,u_i)}\in\Omega,
\end{equation}
where $\text{range}(w^*) = \bigcup\limits_{n\in\Z}w(n)$ for $w^*\in\mathsf W^*$ and any $w\in\pi^{-1}(w^*)$. 
The complement of $\mathcal I^u$ is called the \emph{vacant set (of random interlacement) at level $u$} and is denoted by $\mathcal V^u$. 

\smallskip

Any random interlacement $\mathcal I^u$ is a measurable map from $(\Omega,\mathcal A)$ to $(\Sigma, \mathcal F)$, where $\Sigma$ is the set of all subsets of $V$ and $\mathcal F$ is the sigma-algebra on $\Sigma$ generated by the $\pi$-system $\big\{\{F\in\Sigma\,:\,F\cap K = \emptyset\},\,K\subset V\text{ finite}\big\}$. The law of $\mathcal I^u$ on $(\Omega,\mathcal A, \P)$ is the probability measure $Q^u$ on $(\Sigma,\mathcal F)$ uniquely determined by the identities
\[
Q^u\big[\{F\in\Sigma\,:\,F\cap K = \emptyset\}\big] = \P[\mathcal I^u\cap K = \emptyset] = e^{-u\,\mathrm{cap}(K)},\quad K\subset V\text{ finite}.
\]

Finally, we recall some results from \cite{TT-RI-Connected}.

\begin{proposition}\label{prop:TT}
Let $G$ be a vertex-transitive amenable transient graph. 
Let $\mathrm{Aut}(G)$ be the group of automorphisms on $G$. 
The following statements hold for every $u>0$.
\begin{itemize}
 \item
(\cite[(40)]{TT-RI-Connected}) $\mathrm{Aut}(G)$ is a measure preserving ergodic flow on $(\Sigma,\mathcal F, Q^u)$.
\item
(\cite[Theorem~3.3]{TT-RI-Connected}) $\mathcal I^u$ is connected almost surely.
\end{itemize}
\end{proposition}

\smallskip

\section{Proof of Theorem~\ref{thm:uniqueness-RI}}\label{sec:uniqueness-proof}

Let $u>0$ be fixed and let $N$ be the number of infinite connected components in $\mathcal V^u$. By Proposition~\ref{prop:TT}, $N$ is constant almost surely. 
We prove that $N\in\{0,1\}$ a.s.\ by ruling out separately the two cases $N=k$ for some $2\leq k<\infty$ (in Proposition~\ref{prop:N=k}) and $N=\infty$ a.s.\ (in Proposition~\ref{prop:N=01}). In Lemma~\ref{l:trifurcation}, we prove that $N=\infty$ a.s.\ implies a positive density of trifurcations (see Definition~\ref{def:trifurcation}). Our approach (in the proofs of Proposition~\ref{prop:N=k} and Lemma~\ref{l:trifurcation}) is based on a novel local modification procedure to reroute random walk paths that visit a finite set $K$ through the random interlacements outside of $K$. A key auxiliary result about a connectedness of random interlacements outside any finite set $K$ is given in Proposition~\ref{prop:IKn-connected}. Some common notation that we use in the proofs are collected in Section~\ref{sec:uniqueness-notation}. 

\smallskip

\subsection{Random interlacements outside finite set}\label{sec:uniqueness-notation}

In this section we introduce notation that we use in the proofs of Proposition~\ref{prop:N=k} and Lemma~\ref{l:trifurcation} and prove a result about connectedness of random interlacement outside finite sets, which is key to justify local moditications in Proposition~\ref{prop:N=k} and Lemma~\ref{l:trifurcation}.

\smallskip

Let $u>0$ and finite $K\subset V$ be fixed. 

Let $\iota^u$ be a random interlacement point process at level $u$. 
We decompose $\iota^u$ into the point process $\iota'$ of trajectories that visit $K$ and $\iota''$ of trajectories that do not visit $K$. Note that $\iota'$ and $\iota''$ are independent.

\smallskip

Recall the sampling procedure for interlacement trajectories that visit $K$ from Section~\ref{sec:RIset}.
Let $N_K$ be the number of trajectories in $\iota'$. On the event $\{N_K=n\}$, 
\begin{itemize}
\item
let $(X_i,X_i')$, $1\leq i\leq n$, be the locations of the first and last visits to $K$ of the interlacement trajectories $\iota'$; and
\item
let $(\gamma_i, \widetilde \gamma_i, \gamma_i')$, $1\leq i\leq n$, be the three fragments of the interlacement trajectories $\iota'$, respectively, the time reversal of the part before the first entrance in $K$, the part between the first entrance and the last visit in $K$, and the part after the last visit in $K$. 
\end{itemize}
Given $N_K=n$, $(X_i,X_i')$, $1\leq i\leq n$, are i.i.d.\ with distribution \eqref{eq:sample-1}, and given their locations on $\partial K$, $(\gamma_i, \widetilde \gamma_i, \gamma_i')$, $1\leq i\leq n$, are conditionally independent with law \eqref{eq:sample-2}. 
We define 
\[
\mathcal I_{K,n} = \bigcup\limits_{w^*\in \iota''}\text{range}(w^*)\cup\bigcup\limits_{i=1}^n \text{range}(\gamma_i)\cup\bigcup\limits_{i=1}^n \text{range}(\gamma_i')
\]
and $\mathcal V_{K,n} = V\setminus \mathcal I_{K,n}$. 
(Note that, given $N_K=n$, $\mathcal I^u = \mathcal I_{K,n}\cup\bigcup\limits_{i=1}^n \text{range}(\widetilde \gamma_i)$.)
\begin{proposition}\label{prop:IKn-connected}
For any $u>0$, finite $K\subset V$ and $n\in\N$, given $N_K=n$, $\mathcal I_{K,n}$ is connected almost surely. 
\end{proposition}
\begin{proof}
Let $A$ be the event that $\mathcal I_{K,N_K}$ is connected. We prove that $\P[A]=1$.

\smallskip

For $\omega = \sum\limits_{i\geq 1}\delta_{(w_i^*,u_i)}$ and $0\leq u'<u''$, we denote by $\iota^{u',u''}$ the random interlacement point process with labels between $u'$ and $u''$: 
\[
\iota^{u',u''}=
\iota^{u',u''}(\omega) = \sum\limits_{i\geq 1:\,u_i\in[u',u'']}\delta_{w_i^*}.
\]
Note that $\iota^{u',u''}$ is independent from $\iota^{u'} (= \iota^{0,u'})$ and has the same distribution as $\iota^{u'' - u'}$.

\smallskip

For $\varepsilon>0$, we write the random interlacement point process $\iota^u$ as the sum of independent interlacement point processes $\iota^{u-\varepsilon}+\iota^{u-\varepsilon,u}=:\hat\iota+\check\iota$.

Let $\hat N_K$ be the number of trajectories of $\hat \iota$ that visit $K$. Let $\hat X_i,\hat X_i'\in\partial K$, $1\leq i\leq \hat N_K$, be the locations of the first resp.\ last visit of these trajectories to $K$ and let $\hat\gamma_i$ and $\hat\gamma_i'$ be the time reversed past 
of the trajectories before the first visit in $K$ resp.\ the future of the trajectories after the last visit in $K$. By \eqref{eq:sample-2}, given $\hat N_K$ and $\{(\hat X_i,\hat X_i'),1\leq i\leq \hat N_K\}$, the paths $\hat\gamma_i$ and $\hat\gamma_i'$ are all independent and distributed as simple random walks conditioned on never returning to $K$; furthermore, they are independent from $\check\iota$. 

\smallskip

Since $\P[v\in\check\iota]=c(\varepsilon)>0$, an independent simple random walk (with arbitrary starting point) hits the range of $\check\iota$ almost surely. 
Since simple random walk on $G$ is transient, also an independent simple random walk conditioned on never returning to $K$ hits the range of $\check\iota$ almost surely. 
Thus, each $\hat\gamma_i$ and $\hat\gamma_i'$ hit the range of $\check\iota$ almost surely. 
We denote this event by $\hat A$. 

\smallskip

Let $G$ be the event that none of the trajectories from $\check\iota$ visits $K$. 
Note that $\P[G]\xrightarrow[\varepsilon\to0]{}1$. Furthermore, since the trace of $\check\iota$ is connected almost surely, $\hat A\cap G\subseteq A$. 
Thus, for every $\varepsilon\in(0,u)$, $\P[A]\geq \P[\hat A\cap G] = \P[G]$, which gives $\P[A]=1$. The proof is completed.
\end{proof}
\begin{remark}
An immediate corollary from Proposition~\ref{prop:IKn-connected} is that for every $u>0$ and finite $K\subset V$, $\mathcal I^u\setminus K$ contains exactly one infinite connected component almost surely.
\end{remark}

\subsection{Number of infinite components is $0$, $1$ or $\infty$}

In this section we rule out the case $N=k$ a.s.\ for some $2\leq k<\infty$. 

\begin{proposition}\label{prop:N=k}
$\P[N=k] = 0$ for all $2\leq k<\infty$.
\end{proposition}
\begin{proof}
Assume on the contrary that $N=k$ a.s.\ for some $2\leq k<\infty$. 

Fix a finite $K\subset V$ such that $\mathrm{int}(K)=K\setminus \partial K$ is connected and intersects all $k$ infinite connected components of $\mathcal V^u$ with positive probability. Denote this event by $E_K$.

Fix $n$ such that the probability of event $E_{K,n}=E_K\cap\{N_K=n\}$ is positive. 
On the event $E_{K,n}$, the infinite connected component of $\mathcal V_{K,n}$ is unique and contains $\mathrm{int}(K)$.

By Proposition~\ref{prop:IKn-connected}, $\mathcal I_{K,n}$ is connected almost surely. Thus, every $X_i$ is connected to $X_i'$, for $1\leq i\leq n$, by a path in $\mathcal I_{K,n}\cap K'$ for some large enough finite $K'\supseteq K$. Denote by $E_{K,n,K'}$ the event that $E_{K,n}$ occurs and every $X_i$ is connected to $X_i'$, for $1\leq i\leq n$, inside $\mathcal I_{K,n}\cap K'$. Then $E_{K,n,K'}$ is $\sigma(N_K,\mathcal I_{K,n},\{(X_i,X_i')\}_{1\leq i\leq n})$-measurable and $\P[E_{K,n,K'}]>0$ for all large enough finite $K'\subset V$. Fix such $K'$. There exists $I\subset K'\setminus\mathrm{int}(K)$ and $x_i,x_i'\in\partial K$, for $1\leq i\leq n$, such that 
\[
\P\big[E_{K,n,K'}, \mathcal I_{K,n}\cap K' = I,\,(X_i,X_i')=(x_i,x_i')\text{ for all }1\leq i\leq n\big]>0.
\]
Note that every $x_i$ is connected to $x_i'$ inside of $I$.

Let $G$ be the event that every bridge $\widetilde\gamma_i$ is a simple path in $I$ from $X_i$ to $X_i'$, for $1\leq i\leq n$. Then there exists $c=c(n,K,K')>0$ such that 
\[
\P\big[G\,|\,N_K=n,\,\mathcal I_{K,n}\cap K' = I,\,(X_i,X_i')=(x_i,x_i')\text{ for all }1\leq i\leq n\big]\geq c>0.
\]
Hence 
\[
\P\big[G,\,E_{K,n,K'}, \mathcal I_{K,n}\cap K' = I,\,(X_i,X_i')=(x_i,x_i')\text{ for all }1\leq i\leq n\big]>0.
\]
However, if this event occurs, then $\mathcal V^u=\mathcal V_{K,n}$, which implies that $\mathcal V^u$ contains a unique infinite connected component. Thus, we have shown that $\P[N=1]>0$, which contradicts the initial assumption. The proof is completed. 
\end{proof}

\subsection{Ruling out infinitely many infinite components}

To rule out the possibility of infinitely many infinite connected components in $\mathcal V^u$ we follow the classical Burton-Keane argument with a bit more general notion of trifurcation. 
For $x\in V$, let $B(x,t)$ be the ball of radius $t$ centered in $x$ and denote by $\mathcal C_{x,t}$ the connected component of $x$ in $\mathcal V^u\cap B(x,t)$. 
\begin{definition}\label{def:trifurcation}
Let $t>0$. We say that $x\in V$ is a \emph{$t$-trifurcation} if there is an infinite connected component $\mathcal C$ of the vacant set $\mathcal V^u$, such that 
\begin{itemize}
\item[(a)]
$x\in \mathcal C$;
\item[(b)]
$\mathcal C\setminus \mathcal C_{x,t}$ contains at least $3$ infinite connected components. 
\end{itemize}
\end{definition}
\begin{lemma}\label{l:trifurcation}
Assume that $N=\infty$ a.s.\ Then there exists $t>0$, such that (for any $x\in V$)
\[
\P\big[\text{$x$ is a $t$-trifurcation}\big]>0.
\]
\end{lemma}
\begin{proof}
The proof is very similar to the proof of Proposition~\ref{prop:N=k}. 

\smallskip

Let $x\in V$. Fix a finite $K\subset V$ such that $\mathrm{int}(K)=K\setminus \partial K$ is connected, contains $x$, and intersects at least $3$ infinite connected components of $\mathcal V^u$ with positive probability. Denote this event by $E_K$.
Fix $n$ such that the probability of event $E_{K,n}=E_K\cap\{N_K=n\}$ is positive. 

Since the bridges $\widetilde\gamma_i$ have finite range, if $E_{K,n}$ occurs, then $\mathcal V_{K,n}\setminus K' = \mathcal V^u\setminus K'$ for all large enough finite $K'\supseteq K$; in particular, $\mathcal V_{K,n}$ contains an infinite component $\mathcal C'$, such that $\mathrm{int}(K)\subset \mathcal C'$ and $\mathcal C'\setminus K'$ contains at least $3$ infinite connected components. Let $E_{K,n,K'}$ be the event that 
\begin{itemize}
\item[(a)]
$N_K = n$;
\item[(b)]
there is an infinite connected component $\mathcal C'$ in $\mathcal V_{K,n}$, such that $\mathrm{int}(K)\subset\mathcal C'$ and $\mathcal C'\setminus K'$ contains at least $3$ infinite connected components;
\item[(c)]
for all $1\leq i\leq n$, $X_i$ is connected to $X_i'$ in $\mathcal I_{K,n}\cap K'$.
\end{itemize}
By Proposition~\ref{prop:IKn-connected}, $\P[E_{K,n,K'}]>0$ for all large enough finite $K'\supseteq K$. 
Now, exactly as in the proof of Proposition~\ref{prop:N=k}---by rerouting the bridges $\widetilde\gamma_i$ through $\mathcal I_{K,n}$---, 
we obtain that $\P[E_{K,n,K'}, \mathcal I_{K,n}=\mathcal I^u]>0$. 
Thus, with positive probability, $\mathcal V^u$ contains an infinite connected component $\mathcal C$, such that $x\in\mathcal C$ and $\mathcal C\setminus K'$ contains at least $3$ infinite connected components. 
Call this event $F_{x,n,K'}$. 
We claim that $F_{x,n,K'}$ implies that $x$ is a $t$-trifurcation for all $t$ large enough. 

\smallskip

Assume that $F_{x,n,K'}$ occurs and let $\{\mathcal C_i\}_{i\in I}$ be all $(\geq 3)$ the infinite connected components of $\mathcal C\setminus K'$.
Choose $t$ large enough, so that $\mathcal C\cap K'\subseteq \mathcal C_{x,t}$. 
Since $\mathcal C_i$'s are not connected in $\mathcal C\setminus K'$, they are also not connected in $\mathcal C\setminus \mathcal C_{x,t}$. Thus, $\mathcal C\setminus\mathcal C_{x,t}$ consists of at least $3$ infinite connected components. Hence $x$ is a $t$-trifurcation.

Finally, since $\P[F_{K,n,K'}]>0$, $x$ is a $t$-trifurcation with positive probability for some $t$. The proof is completed. 
\end{proof}

\begin{proposition}\label{prop:N=01}
$\P[N=\infty] = 0$.
\end{proposition}
\begin{proof}
Assume on the contrary that $N=\infty$ a.s.\ and fix a $t$ as in Lemma~\ref{l:trifurcation}. 

\smallskip

By arguing exactly as in the proof of \cite[Theorem~2.4]{HJ-Uniqueness}, we notice that for any finite set of $t$-trifurcations $\mathcal T$ of an infinite connected component $\mathcal C$ such that $\mathcal C_{x,t}\cap\mathcal C_{x',t}=\emptyset$ for all different $x,x'\in\mathcal T$, the set $\mathcal C\setminus\bigcup\limits_{x\in\mathcal T}\mathcal C_{x,t}$ contains at least $|\mathcal T|+2$ infinite connected components. Thus, an infinite component with $j$ $t$-trifurcations in a finite set $W$, which are pairwise at distance at least $2t+1$ from each other, intersects $\partial W'$ in at least $j+2$ vertices, where $W' = \bigcup\limits_{w\in W} B(w,t+1)$. Consequently, using the vertex-transitivity of $G$, the total number $\mathcal T(W)$ of $t$-trifurcations in $W$ is at most $|B(w,2t+1)|(|\partial W'|-2)$, which is at most $|B(w,2t+1)|\,|B(w,t+1)|\,|\partial W| =: c_1|\partial W|$. On the other hand, by Lemma~\ref{l:trifurcation} and the vertex-transitivity of $G$, $\E[\mathcal T(W)] = |W|\,\P[x\text{ is $t$-trifurcation}]=c_2|W|$ for some $c_2>0$. We conclude that for any finite $W\subset V$, 
\[
\frac{|\partial W|}{|W|}\geq \frac{c_2}{c_1}>0,
\]
thus also the vertex isoperimetric constant $\kappa_V(G)$ is positive. 
This contradicts with the assumption that $G$ is amentable. 
The proof is completed. 
\end{proof}


\begin{thebibliography}{99}

\bibitem{BK-Continuity}
J. van den Berg and M. Keane (1984) 
On the continuity of the percolation probability function. 
\emph{Particle Systems, Random Media and Large Deviations} (R.T. Durrett, ed.), Contemporary Mathematics Series, \textbf{26}, AMS, Providence, R. I., 61--65.

\bibitem{BK-Uniqueness}
R.M. Burton and M. Keane (1989) Density and uniqueness in percolation. \emph{Comm. Math. Phys.} \textbf{121}, 501--505.

\bibitem{CT-Book}
J. \v{C}ern\'y and A. Teixeira (2012) \emph{From random walk trajectories to random interlacements}. Ensaios Mathem\'aticos \textbf{23}.

\bibitem{DRS-Book}
A. Drewitz, B. R\'ath and A. Sapozhnikov (2014) \emph{An Introduction to random interlacements}. SpringerBriefs in Mathematics, Berlin. 

\bibitem{DGRS-GFF}
H. Dumini-Copin, S. Goswami, P.-F. Rodriguez and F. Severo (2023) 
Equality of critical parameters for percolation of Gaussian free field level sets. 
\emph{Duke Math. J.} \textbf{172(5)}, 839--913.

\bibitem{HJ-Uniqueness}
O. H\"aggstr\"om and J. Jonasson (2006) Uniqueness and non-uniqueness in percolation theory. 
\emph{Probab. Surveys} \textbf{3}, 289--344.

\bibitem{MS-BI}
Y. Mu and A. Sapozhnikov (2024) On questions of uniqueness for the vacant set of Wiener sausages and Brownian interlacements. \emph{Probab. Theory Related Fields} \textbf{190}, 703--751.

\bibitem{Sznitman-AM}
A.-S. Sznitman (2010) Vacant set of random interlacements and percolation. \emph{Ann. Math.} \textbf{171}, 2039--2087.

\bibitem{Teixeira-RI}
A. Teixeira (2009) Interlacement percolation on transient weighted graphs. \emph{Electron. J. Probab.} \textbf{14}, 1604--1628.

\bibitem{Teixeira-Uniqueness}
A. Teixeira (2009) On the uniqueness of the infinite cluster of the vacant set of random interlacements. 
\emph{Ann. Appl. Probab.} \textbf{19(1)}, 454--466.

\bibitem{TT-RI-Connected}
A. Teixeira and J. Tykesson (2013) Random interlacements and amenability. 
\emph{Ann. Appl. Probab.} \textbf{23(3)}, 923--956.


\end{thebibliography}
\end{document}